\documentclass[12pt]{amsart}
\usepackage{amssymb}
\usepackage{verbatim}
\usepackage[usenames]{color}
\usepackage{url}
\usepackage[all]{xy}

\newtheorem{theorem}{Theorem}[section]
\newtheorem*{theorem*}{Theorem}
\newtheorem{proposition}[theorem]{Proposition} 
\newtheorem{observation}[theorem]{Observation} 
\newtheorem{lemma}[theorem]{Lemma}
\newtheorem{cor}[theorem]{Corollary}
\newtheorem{fact}[theorem]{Fact}
\newtheorem{conjecture}[theorem]{Conjecture}

\theoremstyle{definition}
\newtheorem{definition}[theorem]{Definition}

\newtheorem{question}[theorem]{Question}

\theoremstyle{remark}
\newtheorem{example}[theorem]{Example} 
\newtheorem{remark}[theorem]{Remark}

\newcommand{\defemp}{\textit}

\newcommand{\forces}{\Vdash}
\newcommand{\zf}{\textrm{ZF}}
\newcommand{\ad}{\textrm{AD}}
\newcommand{\pd}{\textrm{PD}}

\newcommand{\zfc}{\textrm{ZFC}}
\newcommand{\ch}{\textnormal{CH}}

\newcommand{\ac}{\textnormal{AC}}
\newcommand{\dc}{\textnormal{DC}}
\newcommand{\psp}{\textnormal{PSP}}

\newcommand{\dom}{\textnormal{Dom}}

\newcommand{\mf}{\mathfrak}
\newcommand{\mc}{\mathcal}
\newcommand{\mbb}{\mathbb}

\newcommand{\p}{\mathcal{P}}

\newcommand{\baire}{{^\omega \omega}}
\newcommand{\bairenodes}{{^{<\omega}\omega}}

\newcommand{\unif}{\textrm{Uniformization}}

\title{Applying Generic Coding with Help to Uniformizations}
\author{Dan Hathaway}

\address{
Dan Hathaway \\
Mathematics Department \\
University of Vermont\\
Burlington, VT 05401, U.S.A.}
\email{Daniel.Hathaway@uvm.edu}

\thanks{A portion of the results of this paper
 were proven during the September 2012 Fields Institute Workshop
 on Forcing while the author was supported by the Fields Institute.
Work was also done while under NSF grant DMS-0943832.
}

\begin{document}

\begin{abstract}
This is a follow up to a paper
 by the author where
 the disjointness relation
 for (the graphs of)
 definable functions from $\baire$ to $\baire$
 is analyzed.
In that paper,
 for each
 $a \in \baire$
 we defined a Baire class one function
 $f_a^{GC} : \baire \to \baire$ which
 encoded $a$ in a certain sense.
Given $g : \baire \to \baire$,
 let $\Psi(g)$ be the statement that $g$
 is disjoint from at most countably many
 of the functions $f_a^{GC}$.
We show the consistency
 strength of $(\forall g)\, \Psi(g)$
 is at most one inaccessible cardinal.
We show that
 $\ad^+$ implies $(\forall g)\, \Psi(g)$. 
Finally, we show that assuming large cardinals,
 $(\forall g)\, \Psi(g)$ holds
 in models of the form
 $L(\mbb{R})[\mc{U}]$ where $\mc{U}$
 is a selective ultrafilter on $\omega$.
\end{abstract}

\maketitle



\section{Introduction}

We do not assume the Axiom of Choice
 in this paper unless explicitly stated.
Our base theory is $\zf$.
In \cite{Hathaway} we isolated
 a lemma about Tree-Hechler Forcing.
We review this as our Lemma~\ref{mainlemma}
 (the so called Main Lemma).
One immediate consequence of this lemma,
 which is the focus of \cite{Sy},
 is the following:
\begin{theorem}[Generic Coding with Help]
If $M$ is a countable transitive model of $\zf$ and
 $x,y \in \mbb{R}$ are reals such that $y \not\in M$,
 then there is some Tree-Hechler generic $G$ over $M$
 such that $x \in L[y,G]$.
\end{theorem}

The proof of the Generic Coding with Help Theorem
 has many interesting consequences
 (which are explored in \cite{Sy}),
 such as the following:
\begin{cor}
Let $M$ be any transitive model of $\zf$.
Let $\bar{a}$ be a set of ordinals not in $M$
 but such that $\mbox{sup}( \bar{a} ) \in M$.
Then there is a $G$ that is set generic over $M$
 (which exists in a class forcing extension of $V$)
 such that $V \subseteq L[G][\bar{a}]$.
\end{cor}

However in this paper we take a step back
 and consider how to apply the Main Lemma
 to one area of descriptive set theory.
Specifically, we apply it to uniformizations.
Recall that given a binary relation
 $R \subseteq \baire \times \baire$
 such that $$(\forall x \in \baire)(\exists y \in \baire)\, (x,y) \in R,$$
 we call $g : \baire \to \baire$ a \textit{uniformization} of $R$
 iff $$(\forall x \in \baire)\, (x,g(x)) \in R.$$
We will consider an edge case of this problem,
 where $R$ is the complement of the graph
 of a function from $\baire$ to $\baire$.
In other words, we have a function $f : \baire \to \baire$
 and the problem is to find a function $g : \baire \to \baire$
 such that $$f \cap g = \emptyset$$
 (the graphs of $f$ and $g$ are disjoint).

We show that there are several definable and uniform ways to
 map every real $a \in \baire$ to a function $f_a : \baire \to \baire$
 such that if $g : \baire \to \baire$ is a ``definable''
 function such that $f_a \cap g = \emptyset$,
 then $a$ is in a countable and canonical set
 of reals associated to $g$
 (or rather, associated to an $\infty$-Borel code for $g$).
We discuss two such mappings:
 one is $a \mapsto f_a^{PSP}$
 which was explained to us by an anonymous referee,
 and a mapping $a \mapsto f_a^{GC}$
 which was developed in \cite{Hathaway}
 but restricted to projective functions $g$ there.

First, let us describe a very coarse version of the problem:
\begin{definition}
A family of functions $\{ f_a : a \in \baire \}$
 from $\baire$ to $\baire$ indexed by $\baire$
 such that $(a,x) \mapsto f_a(x)$ is Borel
 is called a \defemp{Borel family}.
\end{definition}

\begin{definition}
Given functions $f,g : \baire \to \baire$,
 we say that $g$ \defemp{avoids} $f$ iff
 $f \cap g = \emptyset$.
\end{definition}

\begin{definition}
\label{family_avoid}
Fix a family $\mc{F} = \{ f_a : a \in \baire \}$
 of functions from $\baire$ to $\baire$
 and a function $g : \baire \to \baire$.
We say that $g$ \defemp{cannot avoid} $\mc{F}$
 iff $\{ a : f_a \cap g = \emptyset \}$
 is countable.
That is, $g$ cannot avoid $\mc{F}$ iff $g$ can avoid only
 countably many functions in $\mc{F}$.
\end{definition}

\begin{definition}
Let $\Psi$
 be the statement that there is a
 Borel family of functions from $\baire$ to $\baire$
 that no function from $\baire$ to $\baire$
 can avoid.
\end{definition}

\begin{question}
In what models does $\Psi$ hold,
 and what families witness that $\Psi$ holds?
\end{question}

Recall that the perfect set property (PSP)
 says that every uncountable set of reals
 contains a perfect subset.
We will review this in Section~\ref{psp_section}.
The PSP holds in both the Solovay model
 and in any model of $\ad$.
We show that the PSP
 implies that the Borel family $\{ f_a^{PSP} : a \in \baire \}$
 cannot be avoided by any function from $\baire$ to $\baire$.
Hence the PSP implies $\Psi$.

So now we know that $\{ f_a^{PSP} : a \in \baire \}$
 cannot be avoided in the Solovay model
 or in any model of $\ad$.
Similarly, we show by the Main Lemma about Tree-Hechler forcing
 that in the Solovay model
 or in any model of $\ad^+$,
 $\{ f_a^{GC} : a \in \baire \}$
 cannot be avoided by any function from $\baire$ to $\baire$.

On the other hand, we show
 in Corollary~\ref{zfc_implies_not_psi}
 that $\zfc$ implies $\neg \Psi$.
We analyze that proof and
 conjecture in Section~\ref{lowerboundsection}
 that $\dc + \Psi$
 implies $(\forall r \in \baire)$
 $\omega_1$ is inaccessible in $L[r]$.

Thus, here is our conjecture regarding
 consistency strengths:
\begin{conjecture}
The following theories are equiconsistent.
\begin{itemize}
\item[1)] $\zfc + \exists$ an inaccessible cardinal;
\item[2)] $\zf + \dc + \psp$;
\item[3)] $\zf + \dc + \{ f_a^{PSP} : a \in \baire \}$ cannot be avoided;
\item[4)] $\zf + \dc + \{ f_a^{GC} : a \in \baire \}$ cannot be avoided;
\item[5)] $\zf + \dc + \Psi$;
\item[6)] $\zf + \dc + \Psi$
 for only projective $g$'s.
\end{itemize} 
\end{conjecture}
The equiconsistency of 1) and 2) is well-known:
 one direction uses the Solovay model.
We show in Section~\ref{psp_section} that
 $\psp$ implies $\{ f_a^{PSP}: a \in \baire \}$
 cannot be avoided, so
 2) actually implies 3).
The consistency of 1) implies
 the consistency of 4)
 using the Solovay model as we show in
 Section~\ref{upperboundsection}.
Each of 3) and 4) imply 5)
 trivially,
 and 5) trivially implies 6).
Finally, that 6) implies
 there is an inner model with an inaccessible
 cardinal is still a conjecture.

$\ad^+$ is an axiom which implies $\ad$,
 the Axiom of Determinacy,
 and it is open whether $\ad$ implies $\ad^+$.
The axiom $\ad^+$ implies that every set of reals
 (and hence every function from $\baire$ to $\baire$)
 has a so called $\infty$-Borel code
 $C \subseteq \mbox{Ord}$.
We will define $\infty$-Borel codes soon.

We are also interested in the following question:
 given a Borel family such as $\{ f_a : a \in \baire \}$
 and a ``definable'' function $g$,
 what is a canonical description of a countable set $S \subseteq \baire$ such that
 $S$ contains $\{ a \in \baire: f_a \cap g = \emptyset \}$?
For the family $\{ f_a^{PSP}: a \in \baire \}$,
 we have the following:

\begin{theorem}
\label{PSP_adp_preview}
Assume $\ad^+$.
Let $g : \baire \to \baire$ be a function.
Let $Y \subseteq \mbox{Ord}$ be an $\infty$-Borel
 code for $g$.
Then for any $a \in \baire$,
 $$[f_a^{PSP} \cap g = \emptyset] \rightarrow a \in \textnormal{HOD}_{\{Y\}}.$$
\end{theorem}

For the $\{ f_a^{GC} : a \in \baire \}$
 family, we prove a sharper result using the Main Lemma:
\begin{theorem}
\label{GC_adp_preview}
Assume $\ad^+$.
Let $g : \baire \to \baire$ be a function.
Let $Y \subseteq \mbox{Ord}$ be an $\infty$-Borel
 code for $g$.
Then for any $a \in \baire$,
 $$[f_a^{GC} \cap g = \emptyset] \rightarrow a \in L[Y].$$
\end{theorem}

Comparing the past two theorems,
 our intuition is
 that $\{ f_a^{GC} : a \in \baire \}$
 is ``harder to avoid''
 than $\{ f_a^{PSP} : a \in \baire \}$.
We compare versions of the past
 two theorems but for projective
 functions $g : \baire \to \baire$
 in Section~\ref{section_comparing}.

Finally, this work suggests defining
 the following two regularity properties:
 $g : \baire \to \baire$
 is PSP-regular iff it cannot avoid
 $\{ f_a^{PSP} : a \in \baire \}$,
 and it is GC-regular iff it cannot avoid
 $\{ f_a^{GC} : a \in \baire \}$.
Projective Determinacy (PD)
 implies that every projective
 $g : \baire \to \baire$
 is both PSP-regular
 and GC-regular (by Section~\ref{section_comparing}).
By Theorem~\ref{PSP_adp_preview}
 and Theorem~\ref{GC_adp_preview} above,
 if $g : \baire \to \baire$
 is in an inner model of $\ad^+$
 containing all the reals,
 then it is both PSP-regular and
 GC-regular.
However, there may be more
 PSP-regular and GC-regular functions.
That is,
 suppose there is a proper class
 of Woodin cardinals and $\ch$
 holds.
Let $\mc{U}$ be a selective
 ultrafilter on $\omega$.
Now
 $L(\mathbb{R})[\mc{U}]$
 is a generic extension of
 $L(\mathbb{R})$
 (see \cite{Ket} and
 \cite{Farah}).
Using an argument pointed out
 to us by Paul Larson,
 the model $L(\mathbb{R})[\mc{U}]$
 also satisfies the $\psp$.
Thus,
 every $g : \baire \to \baire$
 in $L(\mathbb{R})[\mc{U}]$
 is PSP-regular.
Our last result is Theorem~\ref{barren_thm}
 that states that in this same model
 $L(\mathbb{R})[\mc{U}]$,
 every $g$ is GC-regular as well.

\begin{remark}
There is a different type of information
 that the disjointness relation can capture.
Namely, assume $\ad$.
Fix $\alpha < \Theta$,
 where $\Theta$ is the smallest ordinal that
 $\baire$ cannot be surjected onto.
Then there is a function
 $f : \baire \to \baire$ such that
 if $g : \baire \to \baire$
 is any function that satisfies
 $g \cap f = \emptyset$,
 then $g$ has Wadge rank $> \alpha$.
We can construct $f$ by diagonalizing
 over all functions of Wadge rank $\le \alpha$:
 let $\langle h_x : x \in \baire \rangle$
 be a logically simple enumeration of all
 continuous functions from $\baire \times \baire$
 to $\baire$.
Let $W \subseteq \baire$ be a set
 of Wadge rank $\alpha$.
For each $x \in \baire$,
 if $h_x^{-1}(W)$ is a function,
 define $f(x) := h_x^{-1}(W)(x)$.
Otherwise,
 define $f(x)$ to be anything.
Every Wadge rank $\le \alpha$ function from
 $\baire$ to $\baire$ appears as some
 $h_x^{-1}(W)$.

Since $\baire \cong \baire \sqcup \baire$,
 we may combine this remark with
 Theorem~\ref{infborthm} which we will prove.
That is, assume $\ad^+$.
For every $\alpha < \Theta$
 and for every $a \in \baire$,
 there is a function $f_{\alpha,a} : \baire \to \baire$
 such that whenever $g : \baire \to \baire$
 satisfies $f_{\alpha,a} \cap g = \emptyset$, then
\begin{itemize}
\item[1)] $g$ has Wadge rank $> \alpha$, and
\item[2)] $a \in L[Y]$ for any $\infty$-Borel
 code $Y \subseteq \mbox{Ord}$ for $g$.
\end{itemize}
\end{remark}

\subsection{$\infty$-Borel sets of reals}

Here is a concept we will use several times,
 so let us introduce it now:
\begin{definition}
A set $X \subseteq \baire$
 is $\infty$-Borel iff there is a
 pair $(C,\varphi)$,
 called an $\infty$-Borel code,
 such that $C$ is a set of ordinals
 and $\varphi$ is a formula such that
 $$X = \{ x \in \baire : L[C,x] \models
 \varphi(C,x) \}.$$
A similar definition applies to relations
 $R \subseteq \baire \times ... \times \baire$.
We abuse language
 and call a set $C \subseteq \mbox{Ord}$ an
 $\infty$-Borel code for $X \subseteq \baire$ iff
 there is a formula $\varphi$ such that
 $(C,\varphi)$ is an $\infty$-Borel code for $X$.
\end{definition}

See \cite{Larson_book} for more on $\infty$-Borel sets
 and $\ad^+$ in general.

We do \textit{not} define a function
 $g : \baire \to \baire$ to be $\infty$-Borel
 iff its graph is $\infty$-Borel:
 if $C$ is an $\infty$-Borel code for the graph
 of $g : \baire \to \baire$,
 there is no guarantee that
 $g(x) \in L[C,x]$.
This is the reason for the following
 definition:

\begin{definition}
A function $g: \baire \to \baire$
 is $\infty$-Borel iff there is a
 pair $(C,\varphi)$,
 called an $\infty$-Borel code,
 such that for all $x \in \baire$ and $n,m \in \omega$,
 $$g(x)(n) = m :\Leftrightarrow 
 L[C,x] \models
 \varphi(C,x,n,m).$$
We abuse language and call
 $C \subseteq \mbox{Ord}$ an
 $\infty$-Borel code for $g: \baire \to \baire$ iff
 there is a formula $\varphi$ such that
 $(C,\varphi)$ is an $\infty$-Borel code for $g$.
\end{definition}

We similarly define $\infty$-Borel codes
 for functions $g : \baire \to \baire \times
 [\omega]^\omega$, etc.
We will sometimes be loose and write a code $(C,\varphi)$
 for the graph of $g$, but we will always mean
 the more technical definition.
Note that if $g : \baire \to \baire$
 is $\infty$-Borel with code $C$,
 then $g(x) \in L[C,x]$ for all $x$.
Our strong definition of a function being $\infty$-Borel
 is justified because
 if every $A \subseteq \baire$ is $\infty$-Borel,
 then every $g : \baire \to \baire$ is $\infty$-Borel.

\begin{example}
Consider the function
 $g : \baire \to \baire$
 defined by $g(x) = x'$,
 where $x'$ is the Turing jump of $x$.
This function is $\infty$-Borel
 with $\infty$-Borel code the empty set
 (because $x'$ is definable in $L[x]$).
\end{example}

The following is important for us:
\begin{fact}
\label{ad_countable_is_well_behaved}
Assume $\ad^+$.
Let $X$ be a countable set of reals and
 let $Y$ be an $\infty$-Borel code for
 $X$.
Then $X \subseteq \textnormal{HOD}_{\{Y\}}$.
\end{fact}
\begin{proof}
See Fact 3.3 of \cite{Chan}.
\end{proof}

\section{$\Psi$ is inconsistent with $\zfc$}
\label{section_zfc_implies_not_psi}

Recall that $\unif$ is the fragment
 of the Axiom of Choice that states that
 given any $R \subseteq \baire \times \baire$
 satisfying $(\forall x \in \baire)
 (\exists y \in \baire)\,
 (x,y) \in R$,
 then there is a function
 $u : \baire \to \baire$ such that
 $u \subseteq R$.
We call $u$ a \textit{uniformization} for $R$,
 or say that $R$ is \textit{uniformized} by $u$.
Within this paper,
 we will not assume the Axiom of Choice
 unless explicitly stated.
$\zf$ is our base theory.

\begin{proposition}
\label{unifprop}
$\unif + \Psi$ implies that if
 $S \subseteq \baire$ is uncountable,
 then it can be surjected onto $\baire$
 by a Borel function.
\end{proposition}
\begin{proof}
Because we are assuming $\Psi$,
 fix a Borel family $\{ f_a : a \in \baire \}$
 that no function can avoid.
Fix an uncountable set $S \subseteq \baire$.
For each $x \in \baire$, the function
 $a \mapsto f_a(x)$ is Borel.
We claim that for some $x \in \baire$,
 the function $a \mapsto f_a(x)$
 surjects $S$ onto $\baire$.
Suppose this is not the case.
For each $x \in \baire$,
 the set $Y_x := \baire - \{ f_a(x) : a \in S \}$
 is non-empty.
Apply $\unif$ to get $g : \baire \to \baire$
 such that $(\forall x \in \baire)\, g(x) \in Y_x$.
Then $g$ is disjoint from
 $f_a$ for each $a \in S$.
Since $S$ is uncountable,
 $g$ avoids the family of $f_a$
 functions, which is a contradiction.
\end{proof}

In Section~\ref{psp_section}
 we will recall that if
 an uncountable set $S \subseteq \baire$
 has a perfect subset,
 then $S$ can be surjected onto $\baire$
 by a Borel function.
This is another indication that $\Psi$
 may be related to $\mbox{PSP}$.

It is clear that $\Psi$ is inconsistent with
 $\zfc + \neg\ch$,
 because given any $S \subseteq \baire$
 of size $\omega_1 < 2^\omega$,
 there is a $g$ disjoint from $f_a$ for
 each $a \in S$.
We will now show that $\Psi$
 is inconsistent with $\zfc + \ch$ as well.
By Proposition~\ref{unifprop},
 every uncountable $S \subseteq \baire$
 can be surjected onto $\baire$
 by a Borel function.
Recall that $\mbox{add}(\mc{B})$
 is the smallest size of a collection
 of meager sets of reals whose union is not meager
 (see \cite{Bartoszynski} for more on $\mbox{add}(\mc{B})$,
 where it is called $\mbox{add}(\mc{M})$).
We have $\omega_1 \le \mbox{add}(\mc{B}) \le 2^\omega$.
This next proposition gives us
 our contradiction.
Paul Larson pointed out how to make the diagonalization
 not get stuck by using the meager ideal.

\begin{proposition}
\label{diag_argument}
Assume $\zfc + \mbox{add}(\mc{B}) = 2^\omega$.
Then there exists a size $2^\omega$ set
 $S \subseteq \baire$ that cannot be surjected onto
 $\baire$ by any Borel function.
\end{proposition}
\begin{proof}
Because $\mbox{add}(\mc{B}) = 2^\omega$,
 the union of $< 2^\omega$ meager sets of reals
 is meager.
For each Borel function $h$
 and each $y \in \baire$,
 $h^{-1}(y)$ has the property of Baire,
 so it is either comeager below
 a basic open set or it is meager.
There can be only countably many $y$
 such that $h^{-1}(y)$ is comeager
 below some basic open set, because otherwise
 there would be two that intersect.

We now begin the construction of
 $S = \{ a_\alpha : \alpha < 2^\omega \}$.
Let $\langle h_\alpha : \alpha < 2^\omega \rangle$
 be an enumeration of all Borel functions
 from $\baire$ to $\baire$.
First, pick any $y_0 \in \baire$ such that
 $X_0 := h_0^{-1}(y_0)$ is meager.
This $y_0$ will witness that $h_0$ does not
 surject $S$ onto $\baire$.
Now pick any $a_0 \in \baire - X_0$.

At stage $\alpha < 2^\omega$,
 pick any $y_\alpha \in \baire$
 such that $X_\alpha :=
 h_{\alpha}^{-1}(y_\alpha)$ is meager
 and does not contain any $a_\beta$
 for $\beta < \alpha$.
This is possible because there are
 only $< 2^\omega$ many
 $y$ such that $h_\alpha^{-1}(y)$ contains
 some $a_\beta$ for $\beta < \alpha$,
 and there are only $\omega$ many $y$
 such that $h_\alpha^{-1}(y)$ is not meager.
Then pick $a_\alpha \in \baire
 - \{ a_\beta : \beta < \alpha \}
 - \bigcup_{\beta \le \alpha} X_\beta$.
When the construction finishes,
 the set $S$
 will have size $2^\omega$ and for each
 $\alpha < 2^\omega$,
 $y_\alpha \not\in h_\alpha(S)$.
\end{proof}

\begin{cor}
\label{zfc_implies_not_psi}
$\zfc$ implies $\neg \Psi$.
\end{cor}
\begin{proof}
Assume, towards a contradiction,
 that $\zfc + \Psi$ is consistent.
Let us work within such a model.
We previously gave a quick argument
 that $\neg \ch$ implies $\neg \Psi$,
 so it must be that $\ch$ holds.
Thus $\mbox{add}(\mc{B}) = 2^\omega$ holds.
We also have Uniformization
 (because of the Axiom of Choice)
 and $\Psi$.
Thus the hypothesis of the
 previous two propositions are satisfied.
However, the conclusions
 of these propositions
 contradict one another.
\end{proof}

\begin{remark}
Miller \cite{Miller} has shown that in the
 iterated perfect set model,
 in which
 $\omega_1 = \mbox{add}(\mc{B}) < \omega_2 = 2^\omega$,
 every size $\omega_2$ set $S \subseteq \baire$
 can be surjected onto $\baire$ by a
 \textit{continuous} function.
The iterated perfect set model is obtained
 by starting with a model of $\ch$ and then
 adding $\omega_2$ many Sacks reals
 by a countable support iteration.
This leads us to the following question:
\end{remark}

\begin{question}
Is it consistent with $\zfc$
 that there is a Borel family
 $\{ f_a : a \in \baire \}$
 such that every $g : \baire \to \baire$
 is disjoint from only $< 2^\omega$
 of the $f_a$ functions?
In such a model we would need $\neg \ch$.
\end{question}

\section{$\psp$ implies $\Psi$}
\label{psp_section}

An anonymous referee has pointed out
 a certain family
 $$\{ f_a^{PSP} : a \in \baire \}$$
 which witnesses
 $\Psi$ when we assume the PSP.
We will describe that
 in this section.

Fix a computable bijection from $\omega$ to
 ${^{<\omega} \omega}$ so that
 we may talk about coding a perfect tree
 $T \subseteq {^{<\omega} \omega}$
 by a real $x \in \baire$.
The set of reals through a perfect tree
 can be surjected onto $\baire$
 in a uniform way that is Borel.
We make this precise in the following lemma:

\begin{lemma}
There is a Borel function $E : \baire \times \baire \to \baire$
 such that for each $x \in \baire$,
 if $x$ codes a perfect tree $T_x \subseteq {^{\omega} \omega}$,
 then $$\{ E(x, y) : y \in [T_x] \} = \baire.$$
Hence, there is an $\alpha < \omega_1$
 such that for every perfect set $S \subseteq \baire$,
 there is a ${\bf{\Sigma}}^0_\alpha$ function from $\baire$ to $\baire$
 that surjects $S$ onto $\baire$.
\end{lemma}
\begin{proof}
Fix an $x \in \baire$ that codes a perfect tree $T_x$.
For each $y \in [T_y]$,
 let $D_x(y) \in {^\omega 2}$ be the sequence of $0$'s and $1$'s
 such that for each $n < \omega$,
 $D_x(y)(n)$ specifies whether $y$
 goes through the leftmost child of the $n$-th splitting node of $T_x$
 along $y$ or whether it goes through a different child.
Now for a fixed $y$, $D_x(y)$ can be considered
 as a sequence of $s_0$ many zeros, followed by a one,
  followed by $s_1$ many zeros, followed by a one, etc.
Define $E(x,y) = \langle s_0, s_1, ... \rangle$.
One can verify that $E ``( \{x\} \times [T_x] ) = \baire$
 and also that the function $E$ is Borel.
\end{proof}

\begin{remark}
\label{universal_set_codes}
Let $\mc{X} = (\baire)^n$ for some $n$.
Let $\Gamma$ be a pointclass.
Suppose there is a set $U \subseteq \baire \times \mc{X}$
 such that for each $B \subseteq \mc{X}$ in $\Gamma$,
 there is a $b \in \baire$ such that
 $B = \{ y : (b,y) \in U \}$.
Suppose also that $U$ is in $\Gamma$.
Then we call $U$ a \textit{universal set}
 and we call $b$ a code for $B$.
For each $\alpha < \omega_1$,
 the pointclass $\Sigma^0_\alpha$
 has a universal set,
 so we may talk about codes for $\Sigma^0_\alpha$ sets.
\end{remark}

\begin{definition}
Fix $\alpha < \omega_1$ such that every perfect subset
 of $\baire$ can be surjected onto $\baire$ by a
 ${\bf{\Sigma}}^0_\alpha$ function
 (such an $\alpha$ exists by the previous lemma).
For each $a \in \baire$,
 let $f_a^{PSP} : \baire \to \baire$ be the function
 such that given any $x \in \baire$,
 if $x$ codes a perfect tree $T_x \subseteq \bairenodes$
 together with a ${\bf{\Sigma}}^0_\alpha$ surjection
 $s_x : [T_x] \to \baire$,
 and $a \in [T_x]$, then
 $$f^{PSP}_a(x) := s_x(a).$$
Otherwise, $f^{PSP}_a(x)$ is the zero sequence.
\end{definition}

We will define the $f_a^{GC}$ functions later.
However,
 it will be useful at this time to
 introduce the following notation:
\begin{definition}
\label{D_g_def}
Given $g : \baire \to \baire$,
 $$D_g^{PSP} = \{ a \in \baire : f_a^{PSP} \cap g = \emptyset \}$$
 $$D_g^{GC} = \{ a \in \baire : f_a^{GC} \cap g = \emptyset \}.$$
\end{definition}

The following will be used later:
\begin{lemma}
\label{raw_complexity_of_d}
Let $\Gamma$ be a pointclass
 containing all the Borel sets that
 is also closed under recursive substitutions.
Let $g : \baire \to \baire$ be in $\Gamma$
 in the sense that the ternary relation ``$g(x)(n) = m$''
 is in $\Gamma$.
Let $\{ f_a : \baire \to \baire \}$ be a Borel family.
Let $D_g := \{ a \in \baire : f_a \cap g = \emptyset \}$.
Then $D_g$ is
$\forall^{\baire}\neg \Gamma$.
\end{lemma}
\begin{proof}
A real $a \in \baire$ is in $D_g$ iff
 $$(\forall x \in \baire)(\forall n \in \omega)
 (\forall m \in \omega)[ g(x)(n) = m \rightarrow
 f_a(x)(n) \not= m].$$
\end{proof}

Here is the connection between the PSP
 and the $f_a^{PSP}$ functions:

\begin{lemma}
\label{psp_key}
Fix a function $g : \baire \to \baire$.
Then $D_g^{PSP}$ cannot contain a perfect subset.
\end{lemma}
\begin{proof}
Towards a contradiction,
 fix a perfect tree $T$ such that $[T] \subseteq D_g^{PSP}$.
Let $x \in \baire$ be such that
 $T_x = T$ and $s_x : [T] \to \baire$ is a surjection.
So by definition of the $f_a^{PSP}$ functions, we have
 $\{ f_a^{PSP}(x) : a \in [T] \} = \baire$.
Thus,
 $\{ f_a^{PSP}(x) : a \in D_g^{PSP} \} = \baire$.
In particular $g(x)$ is in this set, so fix
 $a \in D_g^{PSP}$ such that $f_a^{PSP}(x) = g(x)$.
Thus $f_a^{PSP} \cap g \not= \emptyset$,
 which contradicts $a$ being in $D_g^{PSP}$.
\end{proof}

\begin{cor}
\label{psp_implies_psi}
Assume the PSP.
Then for each $g : \baire \to \baire$,
 $D_g^{PSP}$ is countable.
Hence, $\Psi$ holds as witnessed by the family
 $\{ f_a^{PSP} : a \in \baire \}$.
\end{cor}
\begin{proof}
Assume, towards a contradiction,
 that there is some fixed $g$ such that $D_g^{PSP}$
 is uncountable.
Then $D_g^{PSP}$ has a perfect subset $[T]$.
This contradicts the lemma above.
\end{proof}

So the PSP implies each $D_g^{PSP}$ is countable,
 but unfortunately we have no proof that
 PSP implies each $D_g^{GC}$ is countable.
Instead, we have a proof that
 $D_g^{GC}$ is countable
 if either 1) $\ad^+$ holds (see Corollary~\ref{adplus_cor})
 or 2) we are in the Solovay model
 (see Corollary~\ref{gc_solovay_model}).

So it might seem that the family of $f_a^{PSP}$
 functions is strictly better than the family of $f_a^{GC}$
 functions.
However, we have the interesting phenomenon that
 in nearly all instances where we can prove $D_g^{GC}$ to be countable,
 we have a better bound on $D_g^{GC}$ than we do for $D_g^{PSP}$.
We explore this more in Section~\ref{section_comparing}.

Assuming $\ad^+$
 we will prove
 $$D^{PSP}_g \subseteq \mbox{HOD}_{\{Y\}}$$
 whenever $Y \subseteq \mbox{Ord}$
 is an $\infty$-Borel code for $g$.
On the other hand,
 still with $\ad^+$,
 in another section we will prove
 $$D^{GC}_g \subseteq L[Y]$$
 whenever $Y \subseteq \mbox{Ord}$ is an $\infty$-Borel code for $g$.
The rest of the section will focus on the former result.
\begin{proposition}[$\zf$]
\label{infty_borel_code_is_local}
Let $Y$ be a set of ordinals.
Let $A \subseteq \baire$ be $\textnormal{OD}_{\{Y\}}$
 in the model $L(Y,\mathbb{R})$,
 where this model satisfies $\ad^+$.
Then $A$ has an $\infty$-Borel code $S \subseteq \mbox{Ord}$
 in $\textnormal{HOD}_{\{Y\}}$.
\end{proposition}
\begin{proof}
This follows by Theorem 10.2.6
 in \cite{Larson_book}.
\end{proof}

We now have the following:
\begin{theorem}
Assume $\ad^+$.
Let $g : \baire \to \baire$ be a function.
Let $Y \subseteq \mbox{Ord}$ be an $\infty$-Borel
 code for $g$.
Then for any $a \in \baire$,
 $$[f_a^{PSP} \cap g = \emptyset] \rightarrow a \in \textnormal{HOD}_{\{Y\}}.$$
\end{theorem}
\begin{proof}
$D_g^{PSP}$ is $\mbox{OD}_{\{Y\}}$
 in $L(Y, \mathbb{R})$.
So $D_g^{PSP}$ has an $\infty$-Borel code
 $S \subseteq \mbox{Ord}$ in $\mbox{HOD}_{\{Y\}}$
 by Proposition~\ref{infty_borel_code_is_local}.
By the PSP, $D_g^{PSP}$ is countable.
Thus
 $D_g^{PSP} \subseteq \mbox{HOD}_{\{S\}}$
 by Fact~\ref{ad_countable_is_well_behaved}.
Since $S \in \mbox{HOD}_{\{Y\}}$
 we have
 $\mbox{HOD}_{\{S\}} \subseteq \mbox{HOD}_{\{Y\}}$.
Thus $D_g^{PSP} \subseteq \mbox{HOD}_{\{Y\}}$.
This is what we wanted to show.
\end{proof}

\section{Consistency Strength Lower Bound of $\zf + \Psi$}
\label{lowerboundsection}

In Section~\ref{section_zfc_implies_not_psi}
 we gave an argument that $\zfc$
 implies $\neg \Psi$.
That is,
 $\zfc$ implies
that every Borel family $\{ f_a : a \in \baire \}$
of functions from $\baire$ to $\baire$
can be avoided by some function $g : \baire \to \baire$.
Using that argument and
 being careful about the complexity
 of the objects being produced,
 we will show in this section that
 $V = L$ implies
 every Borel family $\{ f_a : a \in \baire \}$
 of functions from $\baire$ to $\baire$
 can be avoided by some ${\bf{\Delta}}^1_2$ function
 $g : \baire \to \baire$.
We will convert this into a conjecture that
 $\zf + \dc$ + ``there exists a Borel family that cannot
 be avoided'' implies that 
  $\omega_1$ is inaccessible
 in $L[r]$ for each $r \in \baire$.

\begin{remark}
Temporarily suppose $\Gamma$ is a pointclass
 closed under quantification of natural
 numbers.
Let $\Delta = \Gamma \cap \neg \Gamma$.
Let $g : \baire \to \baire$.
Consider the ternary relation ``$g(x)(n) = m$''.
Since
 $$g(x)(n) \not= m \Leftrightarrow
 (\exists i \in \omega)\, i \not= m
 \wedge g(x)(n) = i,$$
 we have that the ternary relation
 is in $\Gamma$ iff it is in the dual
 $\neg \Gamma$.
Since
 $$g(x) = y \Leftrightarrow (\forall n \in \omega)[
 (\forall m \in \omega)\, m = y(n) \rightarrow
 g(x)(n) = m],$$
 if the ternary relation ``$g(x)(n) = m$'' is in $\Gamma$
 then the binary relation ``$g(x) = y$'' is in $\Gamma$.
Similarly, since
 $$g(x)(n) = m \Leftrightarrow (\exists y \in \baire)[
 g(x) = y \wedge y(n) = m],$$
 $$g(x)(n) = m \Leftrightarrow (\forall y \in \baire)[
 g(x) = y \Rightarrow y(n) = m],$$
 if the binary relation is in $\Gamma$,
 then the ternary relation is in
 $\exists^\baire \Gamma$ and $\forall^\baire \Gamma$.
By what we said about the binary
 relation versus the ternary relation,
 we have that the following are equivalent:
Now fix an $1 \le n < \omega$.
\begin{itemize}
\item[1)] The binary relation ``$g(x) = y$'' is $\Sigma^1_n$.
\item[2)] The binary relation is $\Pi^1_n$.
\item[3)] The binary relation is $\Delta^1_n$.
\item[4)] the ternary relation ``$g(x)(n) = m$'' is $\Sigma^1_n$
\item[5)] the ternary relation is $\Pi^1_n$
\item[6)] the ternary relation is $\Delta^1_n$.
\end{itemize}


\end{remark}

Using a definition of \cite{mos},
 a well-ordering $\le$ of $\baire$ is called
 $\Gamma$-good iff it is in $\Gamma$
 and whenever $P$ is a binary $\Gamma$-relation,
 then the relations
 $Q(x,y) \Leftrightarrow (\exists x' \le x)\, P(x',y)$ and
 $R(x,y) \Leftrightarrow (\forall x' \le x)\, P(x',y)$
 are in $\Gamma$.
Note that if $\le$ is $\Gamma$-good,
 then it is also $\neg \Gamma$-good.
If $V = L[r]$ for some $r \in \baire$,
 then there is a $\Sigma^1_2(r)$-good
 well-ordering of $\baire$
 of order type $\omega_1$.

We will follow Remark~\ref{universal_set_codes}
 in the construction below.
That is,
 for $\alpha < \omega_1$,
 we will use codes to talk about the $c$-th
 ${\bf{\Sigma}}^0_\alpha$ function
 $h_c : \baire \to \baire$,
 where $c \in \baire$.
That is, we fix a universal ${\bf{\Sigma}}^0_\alpha$
 set and use its sections to get all the
 ${\bf{\Sigma}}^0_\alpha$
 functions from $\baire$ to $\baire$.

\begin{lemma}
\label{least_is_good}
Let $\le$ be a $\Gamma$-good well-ordering of $\baire$.
Let $P \subseteq \baire \times \baire$
 be a binary $\Delta$-relation
 such that $(\forall y \in \baire)(\exists x \in \baire)\, P(x,y)$.
Then the relation $P' \subseteq \baire \times \baire$,
 defined by $P'(x,y) := x$ is the $\le$-least real
 satisfying $P(x,y)$,
 is also a $\Delta$-relation.
\end{lemma}
\begin{proof}
We can assume that the relation $R(a,b) := (a = b)$ is $\Delta$.
$$P'(x,y) = P(x,y) \wedge
 (\forall x' \le x)[ x' \not= x \rightarrow \neg P(x,y)]$$
 is $\Gamma$, because $\neg P(x,y)$ is $\Gamma$
 (because $P$ is $\Delta$)
 and so
 $x' \not= x \rightarrow \neg P(x,y)$ is $\Gamma$
 and so on.
On the other hand,
 $$\neg P'(x,y) := \neg P(x,y) \vee (\exists x' \le x)[x' \not= x
 \wedge P(x',y)]$$
 is $\Gamma$,
 and so $P'$ is $\neg \Gamma$.
Thus $P'$ is $\Delta$.
\end{proof}

\begin{definition}
Fix a computable bijection from $\omega$
 to $\omega \times \omega$.
Given a relation $R \subseteq \omega \times \omega$,
 we may use that bijection to encode
 $R$ as a subset of $\omega$,
 which we can then identity as
 an element of $\baire$.
In this way,
 given a hereditarily countable set $S$,
 call $c \in \baire$ an $H(\omega_1)$ \textit{code} for $S$
 iff $c$ codes a binary relation
 $R \subseteq \omega \times \omega$
 that is isomorphic to the
 $\in$ relation on the transitive closure
 of $S \cup \{S\}$
 such that if we let
 $\pi_c : \langle \omega, R \rangle \to
 \langle \mbox{TC}(S \cup \{S\}), \in \rangle$
 be the isomorphism,
 then $\pi_c(0) = S$.
\end{definition}
Note that the set of all $H(\omega_1)$
 codes is a $\Pi^1_1$ set.
One point of $H(\omega_1)$ codes
 is to convert
 the quantification over the
 elements of a fixed
 hereditarily countable $S$
 to a number quantifier using
 a code for $S$ as a parameter.
Specifically,
 given an $H(\omega_1)$ code $c$
 for a hereditarily countable set $S$,
 the set $S \cap \baire$
 of reals in $S$ is
 $\Delta^1_1(c)$.

\begin{theorem}
Fix $2 \le n < \omega$.
Assume that $\ch$ holds.
Assume also there is a
 ${\bf{\Sigma}}^1_n$-good well-ordering $\le$ of $\baire$
 of order type $\omega_1$.
Fix $\alpha < \omega_1$.
There is an uncountable set $S \subseteq \baire$
 along with a
 ${\bf{\Delta}}^1_n$ function $H : \baire \to \baire$
 such that whenever $c \in \baire$ is a code
 for a ${\bf{\Sigma}}^0_\alpha$ function
 $h : \baire \to \baire$, then
 $H(c) \not\in h``S$.
That is, $H$ witnesses that no
 ${\bf{\Sigma}}^0_\alpha$ function surjects
 $S$ onto $\baire$.
\end{theorem}
\begin{proof}
For each $x \in \baire$,
 by induction on the $\le$-rank of $x$,
 define a pair $(a_x, y_x) \in \baire \times \baire$
 as follows.
Also, note by induction that each pair
 $(a_x,y_x)$ is unique.
\begin{itemize}
\item[1)] $y_x$ is the $\le$-least real such that
\begin{itemize}
\item[1a)] $h_x^{-1}(y_x)$ is meager;
\item[1b)] $(\forall x' \le x)\, x' \not= x
 \Rightarrow a_{x'} \not\in h_x^{-1}(y_x)$.
\end{itemize}
\item[2)] $a_x$ is the $\le$-least real such that
\begin{itemize}
\item[2a)] $(\forall x' \le x)\, x' \not= x
 \Rightarrow a_x \not= a_{x'}$;
\item[2b)] $(\forall x' \le x)\, a_x \not\in
 h_{x'}^{-1}(y_{x'})$.
\end{itemize}
\end{itemize}

Before we proceed, let us
 prove that $(a_x,y_x)$ exists.
Fix $x$ and suppose we have defined
 $(a_{x'},y_{x'})$ for all $x' < x$.
We will find the $(a_x,y_x)$.
First, we will define the $y_x$
 that works.
Since there are only countably many
 $y$'s such that
 $a_{x'} \in h_x^{-1}(y_x)$
 (because $\{ a_{x'} : x' < x \}$ is countable),
 we can just throw out those $y$'s
 when trying to satisfy 1b).
There are only countably many
 $y$'s such that $h_x^{-1}(y)$ is not meager.
This is because $h_x$ is Borel
 so each $h_x^{-1}(y)$ is Borel.
For each $y$ such that $h_x^{-1}(y)$
 is not meager,
 we may pick a non-empty open set $U_y$
 such that the symmetric difference
 between $h_x^{-1}(y)$ and $U_y$ is meager.
Now if there are uncountably many $y$'s
 such that $h_x^{-1}(y)$ is non-meager,
 then there must be two of the corresponding
 $U_y$'s whose intersection is non-empty
 (and open).
A contradiction easily follows.
Thus, we can throw away just countably many $y$'s
 to satisfy 1a).
To summarize, we threw away the countably many $y$'s
 that did not satisfy 1a) and we threw away
 the countably many $y$'s that did not satisfy 1b),
 so we are left with the cocountable set of $y$'s
 that satisfy both 1a) and 1b).
We pick $y_x$ to be the least such $y$.

Now we must pick an $a_x$ that satisfies 2).
Getting $a_x$ to satisfy 2a) is easy because
 $\{ a_{x'} : x' < x \}$ is countable.
That is,
 there is a cocountable
 (and hence comeager) set of $a$'s
 that work for 2a).
For 2b),
 since each $h_{x'}^{-1}(y_{x'})$ is meager,
 and the union of these for $x' \le x$
 is also meager and hence its complement
 is comeager.
Thus the set of $a_x$'s that work for
 2) is comeager, and so we can pick $a_x$
 to be the $\le$-least such one.

Let our set $S$ be $$S = \{ a_x : x \in \baire \}.$$
Note that by 2a), $S$ is uncountable.
Let $H : \baire \to \baire$ be the function
 $$H(x) := y_x.$$
We must now do two things:
 show that $H$ witnesses that no
 ${\bf{\Sigma}}^0_\alpha$ function surjects
 $S$ onto $\baire$,
 and show $H$ is ${\bf{\Delta}}^1_n$.
We will do the former first.

Note that by 1b) and 2b) together, we have
 $$(\forall x_1, x_2 \in \baire)\,
 a_{x_1} \not\in h_{x_2}^{-1}(y_{x_2}).$$ 
That is,
 $$(\forall x_1, x_2 \in \baire)\,
 h_{x_2}( a_{x_1} ) \not= y_{x_2}.$$ 
Now fix a ${\bf{\Sigma}}^0_\alpha$
 function $h : \baire \to \baire$.
We will show that $h$ does not surject
 $S$ onto $\baire$.
Fix an $x_2 \in \baire$
 such that $h_{x_2} = h$.
We now have
 $$(\forall x_1 \in \baire)\,
 h_{x_2}( a_{x_1} ) \not= y_{x_2}.$$
So $y_{x_2}$ is not in the range of
 $h \restriction S$.

The last thing we must do is show
 $H$ is ${\bf{\Delta}}^1_n$.
For each $c \in \baire$,
 let $F_c$ be the function
 $x \mapsto (a_x, y_x)$
 restricted to $\{ x : x \le c \}$.
Since $F_c$ is hereditarily countable,
 it has an $H(\omega_1)$ code.
Let $J : \baire \to \baire$ be the function
 defined by $J(c) :=$ the $\le$-least
 $H(\omega_1)$ code for $F_c$.
Consider the relation
 $R \subseteq \baire \times \baire$
 defined by
 $R(d,c) := $``$d$ is an $H(\omega_1)$
 code for $F_c$''.
We will show that
 $R$ is ${\bf{\Delta}}^1_n$.
It will follow that $J$ is
 ${\bf{\Delta}}^1_n$
 (by the proof of Lemma~\ref{least_is_good}),
 and so $H$ is ${\bf{\Delta}}^1_n$.

Note that the well-ordering
 $\le$ is in fact ${\bf{\Delta}}^1_n$.
In this paragraph we will show that
 the relation $R_1 \subseteq \baire \times \baire$
 defined by $R_1(d,c) :=$
 ``$d \in \baire$ is an $H(\omega_1)$
 code for $\{ x : x \le c \}$''
 is ${\bf{\Delta}}^1_n$.
Quantifying over the reals in the countable
 set coded by a $d$ is a number quantifier,
 not a real quantifier.
So,
 ``$(\forall x \in $ the set coded by
 $d)\, x \le c$'' is ${\bf{\Delta}}^1_n$.
On the other hand,
 ``$(\forall x \le c)\, x \in$ the set coded by
 $d$'' is ${\bf{\Delta}}^1_n$ because
 $\le$ is ${\bf{\Sigma}}^1_n$-good and
 ${\bf{\Pi}}^1_n$-good.
This shows that the relation $R_1$ is ${\bf{\Delta}}^1_n$.
Similarly $R_2(d,c) :=$
 ``$d \in \baire$ is an $H(\omega_1)$ code for a function from
 $\{ x : x \le c \}$ to $\baire \times \baire$''
 is ${\bf{\Delta}}^1_n$.

We will now prove that ``$R(d,c) := d$
 is an $H(\omega_1)$ code for 
 $F_c$ is ${\bf{\Delta}}^1_n$'',
 and this will complete the proof.
Because $\le$ is ${\bf{\Sigma}}^1_n$-good and
 ${\bf{\Pi}}^1_n$-good,
 it suffices to show that 1) and 2) are
 ${\bf{\Delta}}^1_n$.
First, consider 1a).
Each set $h_x^{-1}(y_x)$ is Borel,
 and we can uniformly get a code
 for this set from $x$ and $y_x$.
Given $\beta < \omega_1$,
 whether or not a code for a ${\bf{\Sigma}}^0_\beta$ set
 codes a meager set is certainly
 ${\bf{\Delta}}^1_n$.
Next, since ``$a_{x'} \not\in h_x^{-1}(y_x)$'' is
 ${\bf{\Delta}}^1_n$
 and $\le$ is ${\bf{\Sigma}}^1_n$-good and
 ${\bf{\Pi}}^1_n$-good
 we have that 1b) is
 ${\bf{\Delta}}^1_n$.
So, the conjunction of 1a) and 1b) is
 ${\bf{\Delta}}^1_n$.
The property of being the $\le$-least real that satisfies
 a ${\bf{\Delta}}^1_n$ relation is ${\bf{\Delta}}^1_n$,
 so it follows that 1) is ${\bf{\Delta}}^1_n$.

Now ``$a_x \not= a_{x'}$'' is certainly ${\bf{\Delta}}^1_n$,
 so 2a) is ${\bf{\Delta}}^1_n$
 because $\le$ is ${\bf{\Sigma}}^1_n$-good and
 ${\bf{\Pi}}^1_n$-good.
Similarly, 2b) is ${\bf{\Delta}}^1_n$.
Now the conjunction of 2a) and 2b) is ${\bf{\Delta}}^1_n$,
 and so 2) is ${\bf{\Delta}}^1_n$ as well.
\end{proof}

\begin{cor}
\label{moregood}
Fix $2 \le n < \omega$.
Assume $\ch$ holds
 and there is a ${\bf{\Sigma}}^1_n$-good
 well-ordering $\le$ of $\baire$
 of order type $\omega_1$.
Let $\{ f_a : a \in \baire \}$
 be a Borel family of functions from
 $\baire$ to $\baire$.
Then there is a ${\bf{\Delta}}^1_n$ function
 $g : \baire \to \baire$ that avoids
 the family.
\end{cor}
\begin{proof}
Fix $\alpha < \omega_1$ such that each function
 $a \mapsto f_a(x)$ is ${\bf{\Sigma}}^0_\alpha$.
Let $S$ and $H$ be from the Lemma above.
Define $g : \baire \to \baire$ as follows.
There is a Borel function $x \mapsto c_x$
 such that for each
 $x \in \baire$,
 the real $c_x$ is a code for the function
 $a \mapsto f_a(x)$.
Fix such a function.
Now for all $x \in \baire$,
 we have $H(c_x) \not\in \{ f_a(x) : a \in S \}$.
Define $g(x) := H(c_x)$.
We now have for each $a \in S$
 that $f_a \cap g = \emptyset$.
Thus since $S$ is uncountable,
 $g$ avoids the family $\{ f_a : a \in \baire \}$.
Also,
 one can check that $g$ is in fact
 ${\bf{\Delta}}^1_n$.
\end{proof}

\begin{remark}
Here are some ways to apply the corollary above.
First, $V = L$ implies there is a
 $\Sigma^1_2$-good well ordering of $\baire$
 of order type $\omega_1$.
Going up the large cardinal hierarchy,
 if $\mbox{LC}$ is a large cardinal axiom consistent
 with there being a $\Sigma^1_n$-good
 well-ordering of $\baire$ of order type $\omega_1$
 (for some fixed $n < \omega$),
 then in such a model we have that
 for each Borel family
 $\{ f_a : a \in \baire \}$ of functions
 from $\baire$ to $\baire$,
 there is a ${\bf{\Delta}}^1_n$ function $g : \baire \to \baire$
 that avoids the family.
So for example, assuming there are only finitely many
 Woodin cardinals does not imply that
 every projective function avoids
 $\{ f_a^{PSP} : a \in \baire \}$.
See \cite{Steel} for a discussion of the mouse
 $M^\#_n$ which has $n \in \omega$
 Woodin cardinals
 but at the same time
 a $\Delta^1_{n+2}$-good well-ordering of
 $\mathbb{R}$.
\end{remark}

We would like to show that if there is a Borel
 family of functions that cannot be avoided by
 a projective function,
 then there is an inner model with an
 inaccessible cardinal.
However our argument relies on the following
 conjecture:
\begin{conjecture}
\label{dist_conjecture}
Let $\{ f_a : a \in \baire \}$
 be a Borel family of functions.
Let $x \in \mathbb{R}$ be such that
 $\omega_1 = \omega_1^{L[x]}$.
Let $S \subseteq \baire$ be a set of reals
 that is in $L[x]$ and is uncountable there.
Then if in $L[x]$ there is a $\mathbf{\Delta}^1_2$
 function $g$ that is disjoint from
 $f_a \restriction L[x]$
 for each $a \in S$,
 then there is a projective function
 $g^+ : \baire \to \baire$ that extends $g$ (in $V$)
 that is disjoint from $f_a$ for each $a \in S$.
\end{conjecture}


\begin{theorem}
\label{fixed_lowerboundcor}
Assume $\dc$.
Assume Conjecture~\ref{dist_conjecture}.
Assume there is a Borel family
 $\{ f_a : a \in \baire \}$ such that
 no projective function
 $g : \baire \to \baire$
 can avoid this family.
Then $(\forall r \in \baire)\,
 r$ is inaccessible in $L[r]$.
\end{theorem}
\begin{proof}
Fix a Borel family $\{ f_a : a \in \baire \}$.
Fix $b \in \baire$ such that the function
 $(a,x) \mapsto f_a(x)$ is $\Delta^1_1(b)$.
Since we are assuming $\zf + \dc$,
 the statement $(\forall r \in \baire)\,\omega_1$
 is inaccessible in $L[r]$ is equivalent to
 the statement $(\forall r \in \baire)\,
 \omega_1^{L[r]} < \omega_1$
 \cite{kan}.
We will prove the contrapositive.
That is, fix $r \in \baire$ such that
 $\omega_1^{L[r]} = \omega_1$.
So we also have $\omega_1^{L[r,b]} = \omega_1$.
We will construct a projective function
 that is disjoint from uncountably many of the
 $f_a$ functions
 (so the projective function avoids the family
 of $f_a$ functions).

Note that in $L[r,b]$,
 there is a ${\bf{\Sigma}}^1_2$-good
 well-ordering of $\baire$
 of order type $\omega_1$.
Apply Corollary~\ref{moregood} above in $L[r,b]$
 to get $S \subseteq \baire \cap L[r,b]$
 uncountable (in $L[r,b]$)
 and let $g \in L[r,b]$ be
 ${\bf{\Delta}}^1_2$ (in $L[r,b]$)
 such that $g$ is disjoint (in $L[r,b]$)
 from each $f_a$ for $a \in S$.
By Conjecture~\ref{dist_conjecture}
 fix a projective function
 $g^+ : \baire \to \baire$ (in $V$)
 that is disjoint (in $V$)
 from $f_a$ for each $a \in S$.
Thus, the projective function
 $g^+$ avoids the family
 $\{ f_a : a \in \baire \}$
 which is what we wanted to show.
\end{proof}

\begin{cor}
\label{fixed_lower_bound_for_psi}
Assume $\dc$.
Assume Conjecture~\ref{dist_conjecture}.
Then $\Psi$ implies that
 $(\forall r \in \baire)\,$
 $\omega_1$ is inaccessible in $L[r]$.
\end{cor}

\section{$f_a^{GC}$ and $\mbb{H}$}
\label{Review}

In this section we will
 review the technology of the
 \textit{Generic Coding with Help} method.
A key ingredient is a classical technique
 for generating an infinite subset of $\omega$
 that is computable from every infinite
 subset of itself
 (such a set is called \textit{introreducible}).
We review this first:
\begin{proposition}
Let $X \subseteq \omega$.
There is an infinite $Y \subseteq \omega$
 such that $X$ is computable from every
 infinite subset of $Y$.
Moreover, $Y$ can be taken to be Turing
 equivalent to $X$.
\end{proposition}
\begin{proof}
Let $\chi : \omega \to 2$
 be the characteristic function of $X$.
Let $p_0, p_1, p_2, ...$
 be the increasing enumeration of all
 the prime numbers.
Let $Y \subseteq \omega$ be the set of all
 numbers of the form
 $p_0^{\chi(0)} p_1^{\chi(1)} ... p_n^{\chi(n)}$
 for all $n \in \omega$.
Then $Y$ is as desired.
This encoding trick is sometimes called
 ``stuttering''.
Indeed, we can see that given any
 $m \in Y$,
 by finding the prime factorization of $m$
 we can read off an initial segment of $\chi$.
If we have infinitely many such $m$'s,
 then we can recover all of $\chi$.
It is not hard to see that $X$
 and $Y$ are Turing equivalent.
\end{proof}

Once and for all,
 fix a Borel function
 that maps each real $a \in \baire$
 to an infinite set $A_a \subseteq \omega$
 such that 1) $a$ and $A_a$ are Turing equivalent
 and 2) $A_a$ is computable from every
 infinite subset of itself.
Now for each $a \in \baire$,
 we will define the function $f_a^{GC} : \baire \to \baire$.
\begin{definition}
Fix a computable function
 $\theta : \omega \to \omega$ such that
 $$(\forall m \in \omega)\, \theta^{-1}(m)
 \mbox{ is infinite. }$$
Given an $a \in \baire$,
 let $e_a : \omega \to A_a$ be the strictly increasing
 enumeration of $A_a$.
Let $\eta_a : A_a \to \omega$ be the function
 $\theta \circ e_a^{-1}$.
\end{definition}
Note that for each $m \in \omega$,
 $\eta_a^{-1}(m) \subseteq A_a$ is infinite.
\begin{definition}
The function $f_a^{GC} : \baire \to \baire$
 is defined as follows: 
Given $x = \langle x_0, x_1, ... \rangle \in \baire$,
 let $i_0 < i_1 < ...$ be the indices $i$
 such that $x_i \in A_a$.
Define $f_a^{GC}$ to be
 $$f_a^{GC}(x) := \langle
 \eta_a(x_{i_0}),
 \eta_a(x_{i_1}), ...
 \rangle.$$
If there are only finitely many
 indices $i$ such that $x_i \in A_a$,
 then define $f_a^{GC}(x)$
 to be all $0$'s after these finitely
 many indices.
\end{definition}

\begin{remark}
Note that
 $(a,x) \mapsto f_a^{GC}(x)$
 is Borel.
\end{remark}

To see how the coding works,
 consider a node $t \in \bairenodes$.
Let $n \in \omega$ be the number of
 $l \in \dom(t)$ such that
 $t(l) \in A_a$.
All $x \in \baire$ that extend $t$
 agree up to the first $n$ values of $f_a(x)$,
 but not at the $(n+1)$-th value.
By extending $t$ by one
 to get $t ^\frown k$ for some $k \in A_a$,
 we can decide the $(n+1)$-th value
 of $f_a(x)$
 to be anything we want.
Even if there is a finite set $S$
 of $k$ which we are not allowed to pick,
 we can still create a $t ^\frown k$
 where the $(n+1)$-th value
 of $f_a(x)$ is anything we want.

The poset $\mbb{H}$,
 a variant of Hechler forcing,
 is equivalent to
 the forcing which consists of
 trees $T \subseteq \bairenodes$
 with co-finite splitting after the stem,
 where the ordering $\le$ is reverse inclusion.
We present $\mbb{H}$
 as consisting of pairs $(t,h)$ such that
 $t \in \bairenodes$ and
 $h : \bairenodes \to \omega$,
 where $t$ specifies the stem
 and $h$ specifies where each node
 beyond the stem
 has a final segment of successors.
That is, we have
 $(t',h') \le (t,h)$ iff
 $h' \ge h$ (everywhere domination),
 $t' \sqsupseteq t$, and
 for each $n \in \dom(t') - \dom(t)$,
 $$t'(n) \ge h(t' \restriction n).$$

Given a set $A \subseteq \omega$,
 there is also a stronger ordering
 $\le^A$ defined by
 $(t',h') \le^A (t,h)$ iff
 $(t',h') \le (t,h)$
 and for each $n \in \dom(t') - \dom(t)$,
 $$t'(n) \not\in A.$$
Informally,
 $q \le^A p$ iff $q \le p$
 and the stem of $q$ does not
 ``hit'' $A$ any more than $p$ already does.
We will also use the main lemma
 from \cite{Hathaway},
 which tells us a situation where we can
 hit a dense subset of $\mbb{H}$
 by making a $\le^A$ extension.
By an $\omega$-model we mean a model of $\zf$
 that is possibly ill-founded but
 whose $\omega$ is well-founded
 (and so equal to the true $\omega$).
Moreover, this next lemma
 only needs $M$ to satisfy a fragment of $\zf$.
\begin{lemma}(Main Lemma)
\label{mainlemma}
Let $M$ be an $\omega$-model
 of $\zf$ and
 $D \in \mc{P}^M(\mbb{H}^M)$
 a set dense$^M$ in $\mbb{H}^M$.
Let $A \subseteq \omega$ be infinite and
 $\Delta^1_1$ in every infinite subset of itself,
 but $A \not\in M$.
Then
 $$(\forall p \in \mbb{H}^M)
   (\exists p' \le^A p)\, p' \in D.$$
\end{lemma}

\section{Abstract $f_a^{GC}$ Theorem}

The point of this next theorem is that
 if a model $M$ of $\zf$ can understand
 a function $g : \baire \to \baire$
 on all its generic extensions
 by the $\mathbb{H}$ poset,
 and if $a \in \baire - M$,
 then we can build a real $x$
 that is $\mathbb{H}$-generic over $M$
 such that $f_a(x) = g(x)$.
This is proved using the
 Generic Coding with Help method
 described in the previous section.
\begin{theorem}(ZF)
\label{abstract_thm}
Let $M$ be a transitive model of $\zf$
 such that $\p^M(\mathbb{H}^M)$ is countable.
Let $g : \baire \to \baire$.
Let $\dot{\tau}$ be an $\mathbb{H}^M$-name
 for a function from $\baire \times \omega$ to $\omega$
 such that for every
 $G \subseteq \mathbb{H}^M$ that is
 $\mathbb{H}^M$-generic over $M$,
 if we let $x = \bigcup \{ t : (\exists h)\, (t,h) \in G \}$,
 then $(\forall n \in \omega)\, \dot{\tau}_G(x)(n) = g(x)(n)$.
Then for all $a \in \baire$,
 $$[f_a \cap g = \emptyset]
 \Rightarrow
 a \in M.$$
\end{theorem}
\begin{proof}
Fix $a$ and assume $a \not\in M$.
We must construct an $x \in \baire$
 such that $f_a(x) = g(x)$.
Since $a$ and $A_a$ are Turing equivalent,
 we have $A_a \not\in M$,
 which allows us to apply
 Lemma~\ref{mainlemma}, the Main Lemma.
By hypothesis, $\mc{P}^{M}(\mbb{H}^{M})$ is
 countable, so fix an enumeration
 $\langle D_n \in \mc{P}^{M}(\mbb{H}^{M}) :
 n \in \omega \rangle$
 of the dense subsets of $\mbb{H}^{M}$ in $M$.

We will construct a generic $G$ for
 $\mbb{H}^{M}$ over $M$.
Let $\dot{x}$ be the canonical name for $x$.
The forcing extension will be
 $M[G] = M[x]$.

First, apply Lemma~\ref{mainlemma} to get
 $p_0 \le^{A_a} 1$ such that $p_0 \in D_0$.
Next, apply Lemma~\ref{mainlemma} to get
 $p_0' \le^{A_a} p_0$ and $m_0 \in \omega$ such that
 $$p_0' \forces \dot{\tau}(\dot{x})(0) = \check{m}_0.$$
Now we have that
 $p_0' \le^{A_a} p_0 \le^{A_a} 1$
 and so
 we have that none of the numbers on the stem of $p_0'$
 are elements of $A$.
That is, $p_0'$ has ``not hit $A$ yet'',
 and so our final value of $f_a(x)(0)$ can be anything.
Now if we let $k \in A$, we can extend $p_0'$
 so that the new stem is $\mbox{Stem}(p_0') ^\frown k$
 and this will define $f_a(x)(0)$.
So, extend the stem of $p_0'$ by one to get
 $p_0'' \le p_0'$ in a way to ensure that
 $f_a(x)(0) = m_0$.

Next, apply Lemma~\ref{mainlemma} to get
 $p_1 \le^{A_a} p_0''$ such that $p_1 \in D_1$.
Next, apply Lemma~\ref{mainlemma} to get
 $p_1' \le^{A_a} p_1$ and $m_1 \in \omega$ such that
 $$p_1' \forces \dot{\tau}(\dot{x})(1) = \check{m}_1.$$
Next, extend the stem of $p_1'$ by one to get
 $p_1'' \le p_1'$ in a way to ensure that
 $f_a(x)(1) = m_1$.

Continue like this infinitely.
Since we have constructed a generic $G$ over $M$,
 we have that for each $i < \omega$,
 $$M[x] \models \dot{\tau}_G(x)(i) = m_i.$$
So by the hypothesis on $\dot{\tau}$, we have
 $$g(x)(i) = m_i$$ for all $i$.
On the other hand, we have ensured that
 for each $i < \omega$,
 $$f_a(x)(i) = m_i.$$
Thus, $f_a(x) = g(x)$.
Hence, $f_a$ and $g$ do not have disjoint graphs,
 which is what we wanted to show.
\end{proof}

\section{$\{ f_a^{GC} : a \in \baire \}$ Cannot Be Avoided
 In The Solovay Model}
\label{upperboundsection}

In this section we will show that
 $\{ f_a^{GC} : a \in \baire \}$ cannot be avoided in the
 Solovay model.

\begin{theorem}
Let $M$ be an inner model of $\zfc$ and let
 $\kappa$ be a strongly inaccessible
 cardinal in $M$.
Assume $V = M[G]$ where
 $G$ is generic for the Levy collapse
 of $\kappa$ over $M$.
Fix $C \in {^\omega \mbox{Ord}}$
 and let $g : \baire \to \baire$ be such that
 there is a formula $\varphi$ such that
 for each $x \in \baire$ and $n,m \in \omega$,
 $$g(x)(n) = m \Leftrightarrow \varphi(C,x,n,m).$$
Then for all $a \in \baire$,
 $$[f_a^{GC} \cap g = \emptyset] \Rightarrow a \in M[C].$$
\end{theorem}
\begin{proof}
Given any $x \in \baire$,
 by the factoring of the Levy collapse
 for countable sets of ordinals
 (see Corollary 26.11 in \cite{Jech}),
 $V$ is a generic extension of $M[C,x]$
 by the Levy collapse of $\kappa$,
 and $\omega_1 = \kappa$
 is inaccessible in $M[C,x]$.
Since the Levy collapse is homogeneous,
 for any $x,n,m$ we have
 $$\varphi(C,x,n,m) \Leftrightarrow
 M[C,x] \models 1 \forces
 \varphi(\check{C},\check{x},\check{n},\check{m}).$$
Letting $\tilde{\varphi}(C,x,n,m)$ be the formula
 $1 \forces \varphi(\check{C},\check{x},\check{n},\check{m})$,
 we have
 $$g(x)(n) = m \Leftrightarrow
 M[C,x] \models \tilde{\varphi}(C,x,n,m).$$
This shows that $M[C]$ can understand $g$ on its
 forcing extensions by the $\mathbb{H}^{M[C]}$ forcing.
Note also that $\mc{P}^{M[C]}(\mbb{H}^{M[C]})$ is countable,
 because $\omega_1 = \kappa$ is inaccessible in $M[C]$.
We can now quote Theorem~\ref{abstract_thm}
 using the model $M[C]$
 and we are done.
\end{proof}

Note that in the theorem above,
 $M[C] \cap \baire$ is countable,
 so $g$ can be disjoint from only countably
 many of the $f_a^{GC}$ functions.

\begin{cor}
\label{gc_solovay_model}
Let $\kappa$ be an inaccessible cardinal.
Let $G$ be generic for the Levy collapse of $\kappa$
 over $V$.
Then
 $$\textnormal{HOD}({^\omega \mbox{Ord}})^{V[G]} \models
 (\forall g : \baire \to \baire)\,
 g \mbox{ cannot avoid } 
 \{ f_a^{GC} : a \in \baire \}.$$
\end{cor}

\section{$\ad^+$ Implies
 $\{ f_a^{GC} : a \in \baire \}$ Cannot Be Avoided}

We showed that $\zf + \psp$ implies $\Psi$
 in Section~\ref{psp_section}.
Thus a consistency strength upper bound for $\zf + \Psi$
 is one inaccessible cardinal, because the $\psp$
 holds in the Solovay model.
Also, $\ad$ implies the $\psp$,
 so $\Psi$ holds in all models of $\ad$
 (and hence in all models of $\ad^+$).
The way we showed that the $\psp$ implies $\Psi$
 is by showing that the $\psp$
 implies that
 $\{ f_a^{PSP} : a \in \baire \}$
 cannot be avoided.

The point of this section is to show that
 $\ad^+$ implies that
 $\{ f_a^{GC} : a \in \baire \}$
 cannot be avoided.
Previously in Section~\ref{upperboundsection}
 we showed that
 $\{ f_a^{GC} : a \in \baire \}$
 cannot be avoided in the Solovay model.

We will use the following well known fact.
\begin{fact}
Assume there is no injection
 from $\omega_1$ into $\baire$.
 Let $M$ be an inner model of $\zfc$.
Then $\omega_1^V$ is a strong limit cardinal in $M$.
Since $\beth_\omega$ is
 the first strong limit cardinal
 and $|\mathbb{H}| < \beth_\omega$,
 it follows that
 $\mc{P}^M(\mbb{H}^M)$ is countable (in $V$).
\end{fact}

\begin{theorem}
\label{infborthm}
Assume there is no injection
 from $\omega_1$ into $\baire$.
Let $g : \baire \to \baire$ be $\infty$-Borel
 with code $C \subseteq \mbox{Ord}$.
Then for all $a \in \baire$,
 $$[f_a^{GC} \cap g = \emptyset]
 \Rightarrow
 a \in L[C].$$
Hence, $g$ cannot avoid
 $\{ f_a^{GC} : a \in \baire \}$.
\end{theorem}
\begin{proof}
Use Theorem~\ref{abstract_thm}
 with $M = L[C]$.
To see that the hypotheses are satisfied,
 note that by the nature of $\infty$-Borel codes,
 $M$ can understand $g$ on all of its
 forcing extensions.
\end{proof}

\begin{cor}
\label{adplus_cor}
Assume $\ad^+$.
No function $g : \baire \to \baire$
 can avoid
 $\{ f_a^{GC} : a \in \baire \}$.
\end{cor}
\begin{proof}
$\ad^+$ implies that every set of reals
 is $\infty$-Borel, and hence that every
 $g : \baire \to \baire$ is $\infty$-Borel.
Also $\ad^+$ implies $\ad$,
 which in turn implies there is no injection
 of $\omega_1$ into $\baire$.
\end{proof}

\section{Comparing $f_a^{PSP}$ and $f_a^{GC}$ for Projective $g$}
\label{section_comparing}

In \cite{Hathaway} we showed the following,
 by considering $\omega$-models of $\zf$
 and by showing the contrapositive:
\begin{theorem}
\label{borel_thm}
Fix $c \in \baire$.
Let $g : \baire \to \baire$ be $\Delta^1_1(c)$.
Then for any $a \in \baire$,
 $$[f_a^{GC} \cap g = \emptyset] \Rightarrow
 a \in \Delta^1_1(c).$$
\end{theorem}

Moving up the definability hierarchy
 to $\Delta^1_2(c)$ functions $g$, we showed
 the following.
We will sketch the proof for reference.

\begin{theorem}
\label{inacc_needed}
Fix $c \in \baire$.
Assume $\omega_1$ is inaccessible in $L[c]$.
Let $g : \baire \to \baire$ be $\Delta^1_2(c)$.
Then for any $a \in \baire$,
 $$[f_a^{GC} \cap g = \emptyset] \Rightarrow
 a \in L[c].$$
\end{theorem}
\begin{proof}
Assume that $a \not\in L[c]$.
We will show that $g \cap f_a^{GC} \not= \emptyset$.
By the Shoenfield Absoluteness theorem,
 $L[c]$ can understand $g$
 on all of its forcing extensions.
Thus by Theorem~\ref{abstract_thm}
 we have that for any $a \in \baire$,
 $$[f_a^{GC} \cap g = \emptyset] \Rightarrow a \in L[c],$$
 which is what we want.
\end{proof}

We believe that
 the inaccessible cardinal from Theorem~\ref{inacc_needed}
 can be removed
 (and $L[c]$ need not be countable).
The assumption that $\omega_1$ is inaccessible in $L[c]$
 is only needed to get
 $\mc{P}^{L[c]}(\mbb{H}^{L[c]})$ to be countable.
However, we can always force
 it to be countable and then
 we can attempt to use the
 Shoenfield absoluteness theorem
 to get what we want.

Moving up the projective hierarchy,
 in \cite{Hathaway}
 we showed the following:
\begin{theorem}
\label{mn_theorem}
Fix $c \in \baire$.
Assume $\pd$
 (Projective Determinacy).
Let $g : \baire \to \baire$ be $\Delta^1_n(c)$
 for some $n \ge 3$.
Then $$[f_a^{GC} \cap g = \emptyset] \Rightarrow
 a \in \mc{M}_{n-2}(c).$$
\end{theorem}

The proof of Theorem~\ref{mn_theorem}
 uses that $\mc{M}_{n-2}(c)$ exists,
 that $\omega_1$ is inaccessible in this model,
 and that its forcing extensions
 by Tree-Hechler Forcing $\mbb{H}$
 can \textit{compute}
 $\Sigma^1_{n}(c)$ truth.
Here, $\mc{M}_n(c)$ is a canonical inner model
 with $n$ Woodin cardinals and containing $c$.
The requirement that $\omega_1$ be inaccessible
 is only needed to get the collection of
 dense subsets of
 $\mbb{H}$ in the inner
 model to be countable in $V$.

Note that assuming $\pd$,
 we have that
 $a$ is $\Delta^1_2$ in $c$ and a countable ordinal
 iff $a \in L[c]$.
For $n \ge 3$,
 $a$ is $\Delta^1_{n}$ in $c$ and a countable ordinal
 iff $a \in \mc{M}_{n-2}(c)$
 \cite{Steel}.
Thus, we may succinctly write the following:
\begin{fact}
\label{pd_fact}
Assume $\pd$.
Let $1 \le n < \omega$.
Let $g : \baire \to \baire$ be a $\Delta^1_n(c)$ function
 for some fixed $c \in \baire$.
Then $f_a^{GC} \cap g = \emptyset$
 implies $a$ is $\Delta^1_n$ in $c$
 and a countable ordinal.
\end{fact}

Recall the definitions
 of $D^{PSP}_g$ and $D^{GC}_g$
 from Definition~\ref{D_g_def}.
For $g : \baire \to \baire$ in a (lightface) projective pointclass
 $\Gamma$, the situation is recorded by the following table
 (assuming PD).
The takeaway is that depending on the complexity of $g$,
 sometimes we have a better bound on $D_g^{GC}$,
 and other times we have a better bound on $D_g^{PSP}$.

\begin{table}[h!]
\caption{\label{tab:individual_table}}
  \begin{center}
    \label{tab:table1}
    \begin{tabular}{l|c|r}
      $\Gamma$ & $D_g^{GC}$ bound & $D_g^{PSP}$ bound \\
      \hline
      & & \\
      $\Delta^1_1$ & $D_g^{GC} \subseteq Q_1$ & $D_g^{PSP} \subseteq C_1$ \\
      $\Delta^1_2$ & $D_g^{GC} \subseteq C_2$ & $D_g^{PSP} \subseteq C_3$ \\
      $\Delta^1_3$ & $D_g^{GC} \subseteq Q_3$ & $D_g^{PSP} \subseteq C_3$ \\
      $\Delta^1_4$ & $D_g^{GC} \subseteq C_4$ & $D_g^{PSP} \subseteq C_5$ \\
      $\Delta^1_5$ & $D_g^{GC} \subseteq Q_5$ & $D_g^{PSP} \subseteq C_5$ \\
      ... & ... & ... \\
    \end{tabular}
  \end{center}
\end{table}
For $n$ odd, $C_n$ is the largest countable $\Pi^1_n$ set.
For $n$ even, $C_n$ is the largest countable $\Sigma^1_n$ set,
 which is also that set of all reals that are $\Delta^1_n$
 in a countable ordinal.
For $n$ odd, $Q_n \supseteq C_n$ is the set of all
 reals that are $\Delta^1_n$ in a countable ordinal.
The middle column of the table is by Theorems
 ~\ref{borel_thm}, ~\ref{inacc_needed}, ~\ref{mn_theorem}.

The rightmost column comes from the following argument:
Suppose $n$ is odd for simplicity.
Suppose $g$ is in the pointclass $\Delta^1_n$.
So $g$ is $\Sigma^1_n$.
The set $D^{PSP}_g$ is $\forall^\mbb{R} \neg \Sigma^1_n
 = \Pi^1_n$
 by Lemma~\ref{raw_complexity_of_d}.
Thus if the pointclass $\Pi^1_n$ has the PSP
 and $D_g^{PSP}$ is not contained in the largest countable
  $\Pi^1_n$ set, then $D_g^{PSP}$ must
  have a perfect subset.
By PD we have that $\Pi^1_n$ has the PSP
 and since $D_g^{PSP}$ does not contain a perfect
 subset,
 $D_g^{PSP}$ must be contained in the largest
 countable $\Pi^1_n$ set
 (which is $C_n$).

Note that we can combine the $f_a^{GC}$ and $f_a^{PSP}$
 families together into one to get the best
 of both worlds:
\begin{definition}
Let $\{ f_a^{BOTH} : a \in \baire \}$
 be the family of functions from $\baire$ to $\baire$
 such that for any reals $a,x \in \baire$,
 $$f_a^{BOTH}(0 ^\frown x) = f^{GC}_a(x)$$
 $$f_a^{BOTH}(1 ^\frown x) = f^{PSP}_a(x)$$
 and for $i \ge 2$,
 define $f_a^{BOTH}(i ^\frown x)$
 to be the zero sequence.
\end{definition}

Let $D_g^{BOTH} := \{ a \in \baire :
 f^{BOTH}_a \cap g = \emptyset\}$.
See Table~\ref{tab:both_table}
 for the bounds
 for the $D_g^{BOTH}$ sets for projective $g$:
\begin{table}[h!]
\caption{\label{tab:both_table}}
  \begin{center}
    \label{tab:table1}
    \begin{tabular}{l|c|r}
      $\Gamma$ & $D_g^{BOTH}$ bound \\
      \hline
      & & \\
      $\Delta^1_1$ & $D_g^{BOTH} \subseteq C_1$ \\
      $\Delta^1_2$ & $D_g^{BOTH} \subseteq C_2$ \\
      $\Delta^1_3$ & $D_g^{BOTH} \subseteq C_3$ \\
      $\Delta^1_4$ & $D_g^{BOTH} \subseteq C_4$ \\
      $\Delta^1_5$ & $D_g^{BOTH} \subseteq C_5$ \\
      ... & ... & ... \\
    \end{tabular}
  \end{center}
\end{table}


\section{Functions in $L(\mbb{R})[\mc{U}]$}

The point of this section is to show,
 assuming large cardinals,
 that functions $g : \baire \to \baire$ in models
 of the form $L(\mbb{R})[\mc{U}]$,
 where $\mc{U}$ is a selective ultrafilter on $\omega$,
 cannot avoid $\{ f_a^{GC} : a \in \baire \}$.
Additionally,
 we will recall that such models
 satisfy the $\psp$, and so functions
 in these models also cannot avoid
 $\{ f_a^{PSP} : a \in \baire \}$.

The significance of Lemma~\ref{fancy_lemma}
 that is soon to come
 is that if the PSP holds in an
 appropriate forcing extension on $L(\mathbb{R})$,
 then no $g$ in that forcing extension can avoid
 $\{ f_a^{GC} : a \in \baire \}$.
Note however that Lemma~\ref{fancy_lemma}
 applies to the forcing over $L(\mbb{R})$
 to add a Cohen subset of $\omega_1$.
However, in that forcing extension,
 there is a well-ordering of $\mbb{R}$,
 so $\Psi$ fails there.

The hypothesis of the next lemma follows
 from a proper class of Woodin cardinals.
This is because of the following fact:
\begin{fact}
Assume $\ac$.
Assume there is a proper class of Woodin cardinals.
Then the following hold:
\begin{itemize}
\item[1)] Every set of reals in $L(\mathbb{R})$ is
 universally Baire.
Moreover, for every universally Baire
 set $A \subseteq \mathbb{R}$,
 the model $L(A,\mathbb{R})$ satisfies $\ad^+$
 (Theorem 7.5 of \cite{ult_L})
 and every set of reals in $L(A,\mathbb{R})$
 is universally Baire
 (Theorem 7.4 of \cite{ult_L}).
\item[2)] Every universally Baire binary relation
 on $\baire$ can be uniformized by a universally Baire function
 \cite{Steel_dm}.
\end{itemize}
\end{fact}

Let us briefly discuss part of 1)
 for the interested reader.
Assume that we already have the result that
 a proper class of Woodin cardinals implies
 $\mbox{AD}^{L(\mathbb{R})}$.
Now assume that there is a proper class of
 Woodin cardinals.
Then in every forcing extension
 there is a proper class of Woodin cardinals.
Hence
 $\mbox{AD}^{L(\mathbb{R})}$ holds
 in every forcing extension.
This by \cite{Feng} implies that every set of reals
 in $L(\mathbb{R})$ is universally Baire.
\begin{lemma}
\label{fancy_lemma}
Assume that for each binary relation $E$ on $\baire$
 in $L(\mathbb{R})$,
 $E$ has a uniformization $u$ such that
 $L(u,\mathbb{R}) \models \ad^+$
 (this holds if there is a proper class of Woodin cardinals).
Let $\mbb{Q} \in L(\mbb{R})$ be a forcing
 that does not add reals
 (when forcing over $L(\mathbb{R})$)
 and whose underlying
 set is $\baire$.
Let $\dot{g} \in L(\mbb{R})$ be such that
 $(1 \forces_\mbb{Q} \dot{g} : \baire \to \baire)^{L(\mbb{R})}$.
Then there exists a set of ordinals
 $C \subseteq \mbox{Ord}$ in an inner model
 of $\ad^+$ containing all the reals such that
 $(\forall q \in \mbb{Q})
  (\forall a \in \baire)$
 $$(q \forces_\mbb{Q} f_a^{GC} \cap \dot{g} = \emptyset)^
 {L(\mbb{R})} \Rightarrow
 a \in L[C,q].$$
\end{lemma}
\begin{proof}
Since we can uniformize every binary relation on $\baire$
 that is in $L(\mbb{R})$,
 let $u : \mbb{Q} \times \baire
 \to \mbb{Q} \times \baire$ be such that
 $L(u,\mathbb{R}) \models \ad^+$ and
 $(\forall q \in \mbb{Q})
  (\forall x \in \baire)$,
 if $u(q,x) = (q',y)$, then
 $q' \le q$ and
 $$(q' \forces_\mbb{Q} \dot{g}(\check{x})
 = \check{y})^{L(\mbb{R})}.$$
Since $L(u,\mathbb{R}) \models \ad^+$,
 let $(C,\varphi)$ be an $\infty$-Borel
 code for $u$ in $L(u,\mathbb{R})$.
That is,
 $(\forall q,q' \in \mbb{Q})
  (\forall x,y \in \baire)$
 $$u(q,x) = (q',y) \Leftrightarrow
 L[C,q,x,q',y] \models
 \varphi(C,q,x,q',y).$$
Note that by our convention for
 $\infty$-Borel codes for functions to $\baire$
 or similar ranges,
 if $u(q,x) = (q',y)$, then
 $q',y \in L[C,q,x]$.

Now fix $q \in \mbb{Q}$.
Assume that $a \not\in L[C,q]$.
We will show that
 $\neg (q \forces_\mbb{Q}
 \dot{g} \cap \check{f}_a = \emptyset)^{L(\mbb{R})}$.
We will do this
 by constructing a $q' \le q$
 and an $x \in \baire$ such that
 $(q' \forces_\mbb{Q} \dot{g}(\check{x}) =
 f_a^{GC}(\check{x}))^{L(\mbb{R})}$.
Consider $L[C,q]$.
The $x$ will be generic over this model
 by the forcing $\mbb{H}^{L[C,q]}$.
Then, setting $(q',y) = u(q,x)$,
 we will have
 $(q' \forces_\mbb{Q} \dot{g}(\check{x}) = \check{y})^
 {L(\mbb{R})}$.
At the same time, we will construct $x$ so that
 $f_a^{GC}(x) = y$.

Let $\dot{x}$ be $\mathbb{H}^{L[C,q]}$-name such that
 $1 \forces \dot{x} = 
 \bigcup \{ t : (\exists h)\, (t,h) \in \dot{G} \}$,
 where $\dot{G}$ is the canonical name for the
 generic filter.
That is, $\dot{x}$ is a name for the
 real $x$ we will construct,
 where $x = \{ t : (\exists h)\, (t,h) \in \dot{G} \}$
 where $G$ is the generic filter we construct.
We will now construct $x$ by building
 a generic filter for $\mbb{H}^{L[C,q]}$
 over $L[C,q]$.
Let $\dot{q}', \dot{y} \in L[C,q]$ be such that
 $(1 \forces_\mbb{H}
 \varphi(\check{C},\check{q},\dot{x},
 \dot{q}',\dot{y}
  ))^{L[C,q]}$.
Then, letting
 $q' = (\dot{q}')_x$ and $y = (\dot{y})_x$
 be the valuations of these names with respect
 to the generic $x$,
 we will have
 $L[C,q,x] \models \varphi(
 C, q, x, q', y)$, so
 $u(q,x) = (q', y)$, which implies
 $q' \le q$ and
 $q' \forces_\mbb{Q} (\dot{g}(\check{x}) = \check{y})
 ^{L(\mbb{R})}$.

Let $\langle D_i : i < \omega \rangle$
 be an enumeration of the dense subsets of
 $\mbb{H}^{L[C,q]}$ in $L[C,q]$.
Let $p_0 \le^A 1$ be in $D_0$.
Let $p_0' \le^A p_0$ and $m_0 \in \omega$
 be such that $p_0'$ decides
 $\dot{y}(0)$ to be $m_0$.
That is,
 $( p_0' \forces_\mbb{H} \check{y}(0) = \check{m}_0 )^
 {L[C,q]}$.
Let $p_0'' \le p_0'$ extend the stem of $p_0'$
 by one to ensure that
 $f_a^{GC}(x)(0) = m_0$.

Now let $p_1 \le^A p_0''$ be in $D_1$.
Let $p_1' \le^A p_1$ and
 $m_1 \in \omega$ be such that
 $( p_1' \forces_\mbb{H} \check{y}(1) = \check{m}_1 )^
 {L[C,q]}$.
Let $p_1'' \le p_1'$ extend the stem of $p_1'$
 by one to ensure that
 $f_a^{GC}(x)(1) = m_1$.

Continue this procedure infinitely.
The descending sequence of conditions constructed
 yields a generic ultrafilter $G$ for
 $\mbb{H}^{L[C,q]}$.
By the way $x = (\dot{x})_G$ was constructed,
 we have $f_a^{GC}(x) = m_i$ for all $i < \omega$.
We also have
 $y(i) = m_i$ for all $i < \omega$.
Finally, we have that
 $(q' \forces_\mbb{Q} \dot{g}(\check{x}) = \check{y} )
 ^{L(\mbb{R})}$.
This completes the proof.
\end{proof}

\begin{observation}
Assume that $\textnormal{PSP}$ holds.
Then a forcing extension that does not add reals
 satisfies $\textnormal{PSP}$ iff every uncountable set
 of reals in the extension has an uncountable
 subset in the ground model.
This is because every perfect set of reals in
 the extension is already in the ground model.
\end{observation}

We will use the \textit{tower number}
 for the next lemma.
\begin{definition}
Given $A, B \subseteq \omega$,
 we write $A \supseteq^* B$
 (and say $A$ is a
 \textit{superset mod} finite of $B$)
 iff $B - A$ is finite.
A \textit{tower}
 is a sequence $\langle A_\alpha : \alpha < \lambda \rangle$
 of infinite subsets of $\omega$
 (where $\lambda$ is an ordinal) such that
 $(\forall \alpha < \beta < \lambda)\,
 A_\alpha \supseteq^* A_\beta$.
The tower number $\mathfrak{t}$
 is the length $\lambda$ of the shortest
 tower that cannot be end extended to
 a strictly longer tower.
\end{definition}

In other words,
 the tower number $\mathfrak{t}$
 is the smallest length $\lambda$
 of a tower
 $\langle A_\alpha : \alpha < \lambda \rangle$
 such that there is no infinite
 $B \subseteq \omega$ such that
 $(\forall \alpha < \lambda)\,
 A_\alpha \supseteq^* B$.
See \cite{Blass} for a discussion
 of the tower number.

Paul Larson pointed out this next argument,
 along with using the generic absoluteness
 of the theory of $L(\mbb{R})$.

\begin{lemma}
\label{t_lemma}
Assume $\ac$.
Assume $\omega_1 < \mf{t}$.
Let $\mbb{Q}$ be the $P(\omega)/\mbox{Fin}$ forcing.
Then $(1 \forces_\mbb{Q} \textnormal{PSP})^{L(\mbb{R})}$.
\end{lemma}
\begin{proof}
Fix $\dot{S} \in L(\mbb{R})$ and $q$ such that
 $(q \forces \dot{S} \subseteq \baire$ is uncountable$)^{L(\mbb{R})}$.
We will construct a $q' \le q$ that forces (over $L(\mathbb{R})$)
 that $\dot{S}$ has an uncountable subset in $L(\mathbb{R})$.
By induction, construct (in $V$) a sequence
 $\langle (q_\alpha, b_\alpha) : \alpha < \omega_1 \rangle$
 such that
 1) the $b_\alpha$'s are distinct reals,
 2) the $q_\alpha$'s are decreasing with $q \ge q_0$,
 and 3)
 $(q_\alpha \forces \check{b}_\alpha \in \dot{S})^{L(\mbb{R})}$
 for each $\alpha < \omega_1$.
Every countable initial segment of the sequence
 we are constructing will be in $L(\mathbb{R})$
 (because $L(\mathbb{R})$ contains every
 countable sequence of reals).
Note that we do not get stuck
 at any stage, and so can construct the entire sequence.
However note that the entire (length $\omega_1$)
 sequence may not be in $L(\mathbb{R})$
 because $L(\mathbb{R})$ may satisfy $\ad$
 and hence have no injection of $\omega_1$ into $\mathbb{R}$.

Let $q'$ be a lower bound of the $q_\alpha$'s,
 which exists because they form a
 decreasing, with respect to almost inclusion,
 sequence of infinite subsets of $\omega$,
 and this sequence cannot be maximal
 because $\omega_1 < \mf{t}$.
That is, we construct $q'$ in $V$,
 however it must be in $L(\mathbb{R})$
 because $L(\mathbb{R})$ contains all the reals.
Now let $K = \{ b \in \baire :
 (q' \forces \check{b} \in \dot{S})^{L(\mathbb{R})} \}$.
Note that $K \in L(\mathbb{R})$.
In $V$ we can see that
 $\{ b_\alpha : \alpha < \omega_1 \} \subseteq K$,
 so $K$ is uncountable (in $V$).
But then also $K$ is uncountable in $L(\mathbb{R})$.
Now note that
 $(q' \forces \check{K} \subseteq \dot{S})^{L(\mathbb{R})}$.
Also note that
 $(q' \forces \check{K}$ is uncountable$)^{L(\mathbb{R})}$
 because the forcing does not add any new
 countable sequences of reals.
Hence $q'$ forces that
 $\dot{S}$ has an uncountable subset
 that is in the ground model $L(\mathbb{R})$,
 which by the observation above finishes the proof.
\end{proof}

\begin{theorem}
\label{barren_thm}
Assume $\ac$.
Assume there is a proper class of Woodin
 cardinals.
Let $\mc{U}$ be a selective ultrafilter
 on $\omega$.
Let $g : \baire \to \baire$ be in
 $L(\mbb{R})[\mc{U}]$.
Then $g$ cannot avoid
 $\{ f_a^{GC} : a \in \baire \}$.
\end{theorem}
\begin{proof}
Let $\mbb{Q}$ be the $P(\omega)/\mbox{Fin}$ forcing.
Since there is a proper class of Woodin cardinals,
 the first order theory of $L(\mbb{R})$ cannot be
 changed by any set sized forcing
 (see Theorem 7.22 of \cite{Steel_handbook},
 also Theorem 3.1.12 in \cite{Larson_stat_tower}).
There is a forcing extension of $V$ in which
 $\omega_1 < \mf{t}$.
By Lemma~\ref{t_lemma},
 in that forcing extension we have
 $(1 \forces_\mbb{Q} \textrm{PSP})^{L(\mbb{R})}$.
Thus, in $V$ we have
 $(1 \forces_\mbb{Q} \textrm{PSP})^{L(\mbb{R})}$.

Another consequence of a proper class of Woodin
 cardinals is that an ultrafilter on $\omega$
 is selective iff it is $\mbb{Q}$-generic over $L(\mbb{R})$
 (see \cite{Farah} and \cite{Ket}).
Thus, we will show that every name $\dot{g} \in L(\mbb{R})$
 for a function from $\baire$ to $\baire$ satisfies
 $$L(\mbb{R}) \models
 1 \forces_\mbb{Q} \dot{g} \mbox{ cannot avoid } \{ f_a^{GC} : a \in \baire \}.$$

Towards a contradiction,
 fix $\dot{g} \in L(\mbb{R})$ and $q \in \mbb{Q}$ such that
 $$L(\mbb{R}) \models
 q \forces_\mbb{Q} \{ a : \dot{g} \cap \check{f}_a = \emptyset \}
 \mbox{ is uncountable}.$$
Since $L(\mbb{R}) \models q \forces_\mbb{Q} \textrm{PSP}$,
 by the observation above fix
 a condition $q' \le q$ and
 an uncountable set
 $S \subseteq \baire$ in $L(\mbb{R})$ such that
 for all $a \in S$,
 $$L(\mbb{R}) \models
 q' \forces_{\mbb{Q}}
 [\dot{g} \cap \check{f}_a = \emptyset].$$
Since there is a proper class of Woodin cardinals,
 Apply Lemma~\ref{fancy_lemma} to get
 the $C \subseteq \mbox{Ord}$ described there.
We have
 $$(\forall a \in S)\, a \in L[C,q'],$$
 which is a contradiction because
 since $L[C,q']$ is an inner model of $\zfc$
 inside a model of $\ad$,
 $\baire \cap L[C,q']$ is countable.
\end{proof}

\section{Final Questions}

We close with a few questions.

\begin{question}
Does $\ad$ imply that no $g : \baire \to \baire$
 can avoid $\{ f_a^{GC} : a \in \baire \}$?
\end{question}

More generally, we can ask the following:
\begin{question}
Does $\psp$ imply that no $g : \baire \to \baire$
 can avoid $\{ f_a^{GC} : a \in \baire \}$?
\end{question}

\section{Acknowledgements}

I would like the thank Andreas Blass, Paul Larson,
 Grigor Sargsyan, and Trevor Wilson
 for discussions on this project.
Larson pointed out the arguments for
 Proposition~\ref{diag_argument} and
 Lemma~\ref{t_lemma}.
He also verified that a proper class of Woodin cardinals
 implies every relation on $\baire$ in $L(\mbb{R})$
 can be uniformized in some model of $\ad^+$.
Wilson explained how much truth small forcing
 extensions
 of $\mc{M}_n(c)$ can compute.
He also explained how $\pd$ suffices for
 Fact~\ref{pd_fact},
 instead of $\omega$ Woodin cardinals.
Sargsyan explained how to show that
 countable sets or reals are well-behaved using
 sufficient determinacy axioms.
We also thank the referee for many useful
 suggestions.


\begin{thebibliography}{99}

\bibitem{Blass}
A. Blass,
\textit{Combinatorial cardinal characteristics of the continuum},
in Foreman, M. and
Kanamori A.(Eds.)
Handbook of Set Theory Volume 1,
Springer, Dordrecht, (2010), 395–489.
\bibitem{Bartoszynski}
T. Bartoszynski and
H. Judah.
\textit{Set Theory, On the Structure of the Real Line}.
AK Peters, Wellesley, Massachusetts, 1995.
\bibitem{Chan}
W. Chan and
S. Jackson.
\textit{Cardinality of wellordered disjoint unions of
 quotients of smooth equivalence relations}.
Annals of Pure and Applied Logic, 172 (2021), no.8. 102988.
\bibitem{Farah}
I. Farah.
\textit{Semiselective Coideals}.
Mathematika, vol 45 (1998), 79-103.
\bibitem{Feng}
Q. Feng,
M. Magidor, and
H. Woodin.
\textit{Universally baire sets of reals}.
Set Theory of the Continuum. Mathematical Sciences Research Institute Publications.
(Woodin H. Judah H, Just W., editor), vol. 26, North-Holland, (1992), 203-242.
\bibitem{Sy}
S. Friedman and
D. Hathaway.
\textit{Generic Coding with Help and Amalgamation Failure}.
The Journal of Symbolic Logic, 86 (2021), no 4,
1385-1395.
\bibitem{Hathaway}
D. Hathaway.
\textit{Disjoint Borel Functions}.
Annals of Pure and Applied Logic, 168 (2017), no.8, 1552-1563.
\bibitem{kan}
A. Kanamori.
\textit{The Higher Infinite:
 Large Cardinals in Set Theory from Their Beginnings}.
Berlin: Springer, 2009.
\bibitem{Jech}
T. Jech.
\emph{Set Theory, The Third Millennium Edition,
 Revised and Expanded}. Springer, New York, NY, 2002.
\bibitem{Ket}
R. Ketchersid,
P. Larson, and
J. Zapletal.
\textit{Ramsey Ultrafilters and Countable-To-One Uniformization}.
Topology and Its Applications 213, (2016), 190-198.
\bibitem{Larson_book}
P.Larson.
\emph{Extensions of the Axiom of Determinacy}.
2017. Manuscript in preparation.
\bibitem{Larson_stat_tower}
P. Larson.
\emph{The Stationary Tower: Notes on a course by W. Hugh Woodin}.
University Lecture Series, vol. 32, American
Mathematical Society, Providence, RI, 2004.
\bibitem{Mar}
D. Martin and
J. Steel.
\textit{The Extent of Scales in L(R)}.
In A. Kechris, D. Martin, Y. Moschovakis (eds)
 Cabal Seminar 79-81.
Lecture Notes in Mathematics, vol 1019.
Springer, Berlin, Heidelberg.
\bibitem{Miller}
A. Miller.
\textit{Mapping a Set of Reals Onto the Reals}.
Journal of Symbolic Logic, 48 (1983), 575-584.
\bibitem{mos}
Y. Moschovakis.
\textit{Descriptive Set Theory}.
Amsterdam: North Holland.
1980.
\bibitem{solovay}
R. Solovay.
\textit{On the Cardinality of $\Sigma^1_2$ Sets of Reals}.
Foundations of Mathematics
 (Symposium Commemorating Kurt G\"odel,
 Columbus, Ohio, 1966),
 Springer, New Youk, 1969, pp. 58-73.
\bibitem{Steel_dm}
J. Steel.
\textit{The Derived Model Theorem}.
math.berkley.edu/$\sim$steel/papers/dm.pdf
Unpublished
2008.
\bibitem{Steel_godel}
J. Steel.
\textit{G\"odel's Program}, in
Kennedy, J.
(Ed.)
The Set-Theoretic Multiverse Part IV,
Cambridge University Press,
(2014),
pp 153-179.
\bibitem{Steel}
J. Steel.
\textit{Projectively Well-ordered Inner Models}.
Annals of Pure and Applied Logic.
74 (1995), no. 1, 77-104.
\bibitem{Steel_handbook}
J. Steel.
\textit{An Outline of Inner Model Theory}, in Foreman, M. and Kanamori A.
(Eds.) Handbook of Set Theory Volume 3, Springer, New York, (2010), pp. 1595-1684.
\bibitem{ult_L}
H. Woodin.
\textit{In Search of Ultimate-L, the 19th Midrasha Mathematicae Lectures},
 The Bulletin of Symbolic Logic, 23 (2017), no 1. (March):
  1–109. doi:10.1017/bsl.2016.34.

\end{thebibliography}
\end{document}